\documentclass[11pt]{article}
\usepackage[square,numbers]{natbib}
\usepackage{amsthm}
\usepackage{authblk}
\usepackage{graphicx}
\usepackage{amssymb}
\usepackage{amsmath}

\makeatletter
\newcommand*{\rom}[1]{\expandafter\@slowromancap\romannumeral #1@}
\makeatother
\newtheorem{theorem}{Theorem}[section]

\begin{document}

\title{ Explicit bounds on exceptional zeroes of Dirichlet $L$-functions \rom{2}}
\author{Matteo Bordignon}
\affil[]{School of Science}
\affil[]{University of New South Wales Canberra }
\affil{m.bordignon@student.unsw.edu.au}
\date{\vspace{-5ex}}
\maketitle
\begin{abstract}
This paper improves the upper bound for the exceptional zeroes of Dirichlet $L$-functions with even characters.
The result is obtained by improving on explicit estimate for $L'(\sigma, \chi)$ for $\sigma$ close to unity, using a result on the average of Dirichlet characters, and on the lower bound for $L(1, \chi)$, with computational aid.
\end{abstract}
\section{Introduction}
This paper is a sequel to \cite{Bordignon}, thus we will frequently reference it.
The aim is obtaining an upper bound on real part of the zeroes of 
\begin{equation*}
L(s,\chi):= \sum_{n=0}^{\infty} \chi(n)n^{-s},
\end{equation*} 
with $\chi$ a Dirichlet character and $\Re(s)\in(0,1)$.\\
From the definition of exceptional zero $\beta_0$ in \cite[\S 1]{Bordignon}, see also \cite{McCurley, Trudgian, Platt}, we can focus on real zeroes of non-principal real characters $\chi \pmod{q}$, with $q \geq 4\cdot 10^5$.
For any such $\beta_0$ we have $\beta_0 \leq 1 -\frac{\lambda}{  q^{1/2}\log^2 q}$, with  $\lambda$ explicit. We list some results below.
\begin{enumerate}
\item Liu and Wang prove $\lambda \approx 6$ for $q> 987$ in \cite[Theorem 3]{Liu-Wang},
\item Ford et al.\ prove $\lambda \approx 19 $ for $q> 10^4$ in \cite[Lemma 3 ]{Ford},
\item Bennett et al.\ prove $\lambda = 40$ for $q> 4\cdot 10^5$ in \cite[Proposition 1.10]{Bennett},
\item The author proves $\lambda = 80$ for $q> 4\cdot 10^5$ in \cite[Theorem 1.3]{Bordignon}.
\end{enumerate}
We can note, from \cite[Theorem 1.3]{Bordignon}, that restricting the above results to odd characters we obtain a significantly better result, thus focusing on even characters will improve the overall result.\\
The above results follow from the mean-value theorem, a lower bound for $L(1,\chi)$, obtained using the Class Number Formula, and an upper bound for $\left| L'(\sigma,\chi)\right|$, with $\sigma \in (\beta_0,1)$.
Liu and Wang obtain the result by dividing the sum for $L'(\sigma, \chi)$ in two, and using that $|\chi(n)|\in\lbrace 0, 1\rbrace$ on the first half and P\'olya--Vinogradov on the second, and a classic lower bound for $L(1, \chi)$ obtained from the Dirichlet Class Number Formula.
Ford et al.\ and Bennett et al.\ improve the results using more precise results and extensive computations.
The author, in \cite{Bordignon}, proves a general result that allows to remove one of the two terms in the upper bound of $L'(\sigma, \chi)$.\\
It is interesting to note that, using the above techniques, we have  
\begin{equation*}
\left| L'(\sigma,\chi)\right| \leqslant \left (\frac{1}{8}+o(1)\right) \log^2 q.
\end{equation*} 
The difference in strength of the above results is in the size of the reminder therm. We will now introduce a different technique, following from a paper of Hua \cite{Hua} on the average of Dirichlet characters, that will allow us to remove the reminder term for even characters and thus obtain an ``optimal" upper bound.
From Theorem \ref{L}, assuming the exceptional zero near the unity, we are able to obtain better upper bounds for $\left| L'(\sigma,\chi) \right|$.

\begin{theorem}
\label{40}
Assume $\chi$ is an even primitive real character and $\sigma\in (\beta_0, 1)$. 
With $\beta_0~\geqslant~ 1-~\frac{100}{ \sqrt{q}\log^2 q}$ and $q> 4\cdot 10^5$, the following bound holds 
\begin{equation}
\label{5}
\left| L'(\sigma,\chi)\right|  \leqslant \frac{1}{8} \log^2 q.
\end{equation} 
\end{theorem}
We will then improve on Bennett et al.'s lower bound for $L(1,\chi)$.
\begin{theorem}
\label{411}
Assume $\chi$ is an even primitive real character. 
With $ q >~4\cdot~10^5 $, the following bound holds
\begin{equation}
\label{833}
L(1,\chi) \geq \frac{12.52}{\sqrt{q}}. 
\end{equation}
\end{theorem}
 These results will give the following upper bounds for $\beta_0$.
\begin{theorem}
\label{41}
Assume $\chi$ is an even non-principal real character. 
With $ q > 4\cdot 10^5 $, the following bound holds
\begin{equation}
\label{83}
\beta_0 \leq 1-\frac{100}{ \sqrt{q}\log^2 q}. 
\end{equation}
\end{theorem}
In \S \ref{s1} we prove Theorem \ref{411}, in \S \ref{s2} Theorem \ref{40}, these two results together will give Theorem \ref{41}. We will conclude proving a more precise version of Theorem \ref{41}.

\section{Upper bound for the exceptional zero $\beta_0$ }
\label{02}
Using the same standard trick as in \cite[\S 3]{Bordignon}, we see that
\begin{equation}
\label{2}
1-\beta_0=\frac{L(1,\chi)}{| L'(\sigma,\chi)|},
\end{equation}
for some $\sigma \in (\beta_0,1)$. Thus we are left to obtain a lower bound for $L(1,\chi)$ and an upper bound for $| L'(\sigma,\chi) |$ for $\sigma \in (\beta_0,1)$.
\subsection{Lower bound for $L(1,\chi)$}
\label{s1}
We start fixing $q > 4\cdot 10^5$.
We use that every real primitive character can be expressed using the Kroneker symbol, as $\chi(n)= (\frac{d}{n})$, with $q= \left| d\right| $.
We consider $d>0$. Dirichlet's Class Number Formula gives
\begin{equation}
\label{77}
L(1,\chi)=\frac{ h(\sqrt{d} )\log \eta_d}{  \sqrt{ d }},\quad \text{with}~ \chi(-1)=1,
\end{equation}
where $\eta_d=(v_0+u_0\sqrt{d})/2$ , with $v_0$ and $u_0$ the minimal positive integers satisfying $v_0^2-du_0^2=4$.
From A.10. in \cite{Bennett} we have that 
\begin{equation}
\label{7}
h(\sqrt{d} )\log \eta_d > 79.2177
\end{equation} when $4\cdot10^5\leqslant d \leqslant 10^7$.
Bennett et al.\ then compute that for all $(d, u_0)$, with $d > 10^7$ and $du_0^2< 2.65\cdot10^{10}$, we have $h(\sqrt{d} )\log \eta_d >417$.
Using their Sage \cite{sage} code and a longer computational time, we compute that for all $(d, u_0)$, with $d > 10^7$ and $  du_0^2\leq 7.5\cdot10^{10}$, we have $h(\sqrt{d} )\log \eta_d >412$. For this computation we used 1000 CPU for a total of approximately 1800 CPU hours. Calculations were performed on Raijin, a high-performance computer managed by NCI Australia.\\
Finally, remembering that $h(\sqrt{d})\geq 1$ and $\eta_d=(v_0+u_0\sqrt{d})/2$, we obtain for all $d > 10^7$ such that $ du_0^2\geq 7.5\cdot10^{10}$
\begin{equation}
\label{8}
 h(\sqrt{d} )\log \eta_d \geq\log u_0 \sqrt{d}  \geq \frac{1}{2}\log (7.5\cdot 10^{10}) \geq  12.52.
\end{equation}
Thus Theorem \ref{411} follows from (\ref{77}), using (\ref{7}) and (\ref{8}).\\
It is interesting to note that in order to improve the bound in (\ref{8}) we have to exponentially increase the range of $du_0^2$, this will make the computational time also increase exponentially.
\subsection{Upper bound for $\left| L'(\sigma,\chi) \right|$ and proof of Theorem \ref{41}}
\label{s2}
The main result used is the following one, that is Theorem 1 in \cite{Louboutin}, with the left-hand side sum starting from $4$. Note that this was not done in \cite{Bordignon} as a negative term would compensate for the exceeding positive terms.
\begin{theorem}
\label{L}
Take $\chi$ a even primitive Dirichlet character, with conductor $q$. Let $A:=\lfloor \sqrt{q}\rfloor-1$. Let $f$ be defined in $[4,\infty)$, $\searrow 0$ and such that $f(n)-2f(n+1)+f(n+2)\geq 0$ for all $n \geq 4$. Then, with $0 \leq \theta \leq 1$, 
\begin{equation*}
\left| \sum_4^{\infty} \chi(n) f(n)\right| \leq \left( \sum_4^A f(n) \right)-\frac{A}{2}f(A)+\frac{A}{2}\left\{f(A)-f(A+1)\right\}+\frac{1}{2}f(A+1)
\end{equation*}
\begin{equation*}
+\frac{\theta}{2}\left\{(A+1)(f(A+1)-f(A+2))+f(A+2)\right\}+ 18f(4)-12f(5).
\end{equation*}
\end{theorem}
\begin{proof}
We will follow the proof of \cite[Theorem 1]{Louboutin}.\\
Define $S(n)=\sum_{a=1}^n \sum_{k=1}^a \chi(k)$.
We have 
\begin{equation*}
\chi(n)=S(n)-2S(n-1)+S(n-2),
\end{equation*}
and, with $a_n =f(n)-2f(n+1)+f(n+2)$, this gives
\begin{equation}
\label{eq1}
\sum_{n\ge k} f(n) \chi(n)= \sum_{n\ge k} a_n S(n) -2f(k)S(k-1)+f(k)S(k-2)+f(k+1)S(k-1).
\end{equation}
For $k\leq A$,
\begin{equation*}
\sum_{n\ge k} a_n S(n) = \sum_{k}^A a_n S(n)+\sum_{n>A} a_n S(n).
\end{equation*}
Now, using that $S(n)\le \frac{n(n+1)}{2}$ and $f(n)-2f(n+1)+f(n+2)\geq 0$, we have
\begin{equation*}
\left| \sum_{k}^A a_n S(n)\right|\leq  \sum_{k}^A \frac{n(n+1)}{2} a_n =  \sum_{k+2}^A  f(n)+\frac{(k+1)(k+2)}{2}f(k+1)+
\end{equation*}
\begin{equation*}
+\frac{k(k+1)}{2}f(k)-2 \frac{k(k+1)}{2}f(k+1) -\frac{A(A+3)}{2}f(A+1)+\frac{A(A+1)}{2}f(A+2)
\end{equation*}
The above, together with (\ref{eq1}), gives
\begin{equation*}
\left| \sum_{n\ge k} f(n) \chi(n) \right|\le   \sum_{k+2}^A  f(n)+\frac{(k+1)(k+2)}{2}f(k+1)+\frac{k(k+1)}{2}f(k) +
\end{equation*}
\begin{equation*}
-2 \frac{k(k+1)}{2}f(k+1)+\left|-2f(k)S(k-1)+f(k)S(k-2)+f(k+1)S(k-1)\right|+
\end{equation*}
\begin{equation*}
-\frac{A(A+3)}{2}f(A+1)+\frac{A(A+1)}{2}f(A+2)+\left| \sum_{n>A} a_n S(n)\right|.
\end{equation*}
Now, with $k=4$, the result follows as in Louboutin's proof.
\end{proof}
Here we assume $\sigma \in (\beta_0, 1)$ and $\beta_0 \geqslant 1-\frac{c}{ \sqrt{q}\log^2 q}$, with $c \in [100,1000]$, to be chosen later, and $q \geq 4\cdot 10^5$.\\
Now we apply Theorem \ref{L} to the function
\begin{equation*}
f(n)=\frac{\log n}{n^{\sigma}},
\end{equation*}
that, for $n\geq 4$, results decreasing and such that $f(n)-2f(n+1)+f(n+2)\geq 0$, as $f(4)-2f(4+1)+f(4+2)\geq 0$ and $f$ is convex for $n \geq 5$.
We denote with $R(A,\sigma)$ the term in the right hand side of the formula in Theorem \ref{L}.
With $d \leq A$, we further obtain by partial summation
\begin{equation*}
\sum_1^A \frac{\log n}{n}\leq \frac{1}{2}\log^2 A  - \frac{1}{2}\log^2 d +\sum_{n=2}^{d} \frac{\log n}{n},
\end{equation*}
and fixing $d= 100$ 
 \begin{equation*}
 - \frac{1}{2}\log^2 d +\sum_{n=2}^{d} \frac{\log n}{n}< 0.
\end{equation*}
This number is so small that we omit it in what follows.
Thus
\begin{equation}
\label{ub}
\left| L'(\sigma, \chi)\right| \leq q^{\frac{1-\beta_0}{2}}\frac{1}{8}\log^2 q+R(A, \sigma).
\end{equation}
Remembering $\beta_0 \geqslant 1-\frac{c}{ \sqrt{q}\log^2 q}$ and choosing $c=100$  it is easy to see, for all $q \geq  4\cdot 10^5$ and $\sigma \in (\beta_0, 1)$, that
\begin{equation*}
\left| L'(\sigma,\chi)\right|  \leqslant \frac{1}{8} \log^2 q,
\end{equation*}
this proves Theorem \ref{40}.
Now Theorem \ref{41} follows easily. We just need to prove the theorem for primitive real characters, indeed if $\chi \pmod{q}$ is induced by some primitive real character $\chi' \pmod{q'}$, then the primitive case yields
\begin{equation*}
\beta_0 \leq 1-\frac{\lambda}{\sqrt{q'}\log^2 q'}\leq 1-\frac{\lambda}{\sqrt{q}\log^2 q}.
\end{equation*}
Thus Theorem \ref{41} follows from (\ref{2}), Theorem \ref{40} and Theorem \ref{411}.\\
We can conclude proving a more ``precise" version of theorem \ref{41}. From (\ref{2})-(\ref{ub}), we obtain $\beta_0 \leq 1-\frac{c}{ \sqrt{q}\log^2 q}$, with the following values for $c$ and ranges for $q\geq 4\cdot 10^5$.
\begin{center}
    \begin{tabular}{| l | l | }
    \hline
    $  q $ & $c$ \\ \hline
    $ q \leq 7\cdot 10^5$ & $624$  \\ \hline
    $ 7\cdot 10^5 \leq q \leq 10^6$ & $ 636$  \\ \hline
    $  10^6 \leq q \leq 3\cdot10^6$ & $641$  \\ \hline
    $ 3\cdot 10^6 \leq q \leq 8\cdot10^6$ & $ 654$  \\ \hline
    $ 8\cdot 10^6\leq q \leq 10^7 $ & $660$  \\ \hline

    \end{tabular}
    \quad
    \begin{tabular}{| l | l | }
    \hline
    $  q $ & $c$ \\ \hline
    $q \leq 10^{12}$ & $105$  \\ \hline
    $q \leq  10^{18}$ & $104$  \\ \hline
    $q \leq 10^{26}$ & $103$  \\ \hline
    $q \leq 10^{43}$ & $102$   \\ \hline
    $q \leq 10^{100}$ & $101$   \\ \hline

    \end{tabular}
   
\end{center}
Note that the drastic decrease of $c$ when $q > 10^7$ is due to the difference between (\ref{7}) and (\ref{8}). 

\section*{Acknowledgements} 
I would like to thank my supervisor Tim Trudgian for his kind help and his sharp suggestions in developing this paper, Prof.\ Ryotaro Okazaki for the suggestion to read Hua's paper \cite{Hua}, Prof.\ Olivier Ramar\'e for the interesting comments and Alberto Sanchez Muzas for the help with the computational part. I would also like to thank the NCI and UNSW Canberra for the computational time. This research was undertaken with the assistance of resources and services from the National Computational Infrastructure (NCI), which is supported by the Australian Government.

\end{document}